\documentclass[a4paper,10pt]{amsart}
\usepackage{amssymb, amsmath, graphicx, amsthm, enumitem, multicol,amsfonts}
\usepackage[margin=1in]{geometry}
\usepackage{cite}
\usepackage{ bbold }
\usepackage{longtable}
\usepackage{booktabs}
\usepackage{placeins}
\usepackage{hyperref}
\usepackage{tikz-cd}
\usepackage{mathabx}

\usepackage[utf8]{inputenc}

\DeclareMathOperator{\C}{\mathbb{C}}

\DeclareMathOperator{\Z}{\mathbb{Z}}
\DeclareMathOperator{\N}{\mathbb{N}}
\DeclareMathOperator{\F}{\mathbb{F}}

\DeclareMathOperator{\cB}{\mathcal{B}}
\newcommand{\cK}{{\mathcal{K}}}
\DeclareMathOperator{\cS}{\mathcal{S}}

\newcommand{\supp}{\mathrm{supp}}

\newcommand{\sstab}{G_N}
\newcommand{\pstab}{\hat G_N}
\newcommand{\Grem}[1]{G(#1)}

\newcommand{\GL}{\mathrm{GL}}
\newcommand{\Vi}{\bar V}
\newcommand{\vK}{\mathbb{K}}
\newcommand{\dK}{{K}}

\newcommand{\ba}{\begin{align}}
\newcommand{\ea}{\end{align}}
\newcommand{\bea}{\begin{eqnarray}}
\newcommand{\eea}{\end{eqnarray}}
\newcommand{\be}{\begin{equation}}
\newcommand{\ee}{\end{equation}}
\newcommand{\wt}{\mathrm{wt}}

\newcommand{\ie}{{\it i.e.~}}
\newcommand{\eg}{{\it e.g.~}}

\newcommand{\diag}{{\mathrm{diag}}}

\newcommand{\vac}{|0\rangle}

\newlength{\slength}
\newcommand*{\bb}{\makebox[\slength][c]{$\bullet$}}
\newcommand*{\wb}{\makebox[\slength][c]{$\circ$}}

\title{The Large $N$ Limit of Orbifold Vertex Operator Algebras}
\author{Thomas Gem\"unden}
\address{Thomas Gem\"unden, Department of Mathematics, ETH Zurich
	CH-8092 Zurich, Switzerland}
\email{thomas.gemuenden@math.ethz.ch}
\author{ Christoph A.~Keller}
\address{Christoph A. Keller,  Department of Mathematics, University of Arizona, Tucson, AZ 85721-0089, USA}
\email{christoph.keller@math.arizona.edu}

\theoremstyle{plain}
\newtheorem{thm}{Theorem}[section]
\newtheorem{lem}[thm]{Lemma}
\newtheorem{prop}[thm]{Proposition}
\newtheorem{cor}[thm]{Corollary}

\theoremstyle{definition}
\newtheorem{defn}{Definition}[section]

\theoremstyle{remark}

\begin{document}

\begin{abstract}
We investigate the large $N$ limit of permutation orbifolds of vertex operator algebras. To this end, we introduce the notion of nested oligomorphic permutation orbifolds and discuss under which conditions their fixed point VOAs converge. We show that if this limit exists, then it has the structure of a vertex algebra. Finally, we give an example based on $\GL(N,q)$ for which the fixed point VOA limit is also the limit of the full permutation orbifold VOA.

\end{abstract}

\maketitle

\section{Introduction}
\subsection{Overview}
Vertex Operator Algebras (VOA) and their conformal field theories (CFT) play an important in the AdS/CFT correspondence \cite{Maldacena:1997re,Aharony:1999ti}. The correspondence conjecturally maps theories of quantum gravity to certain types of VOAs.
Such VOAs tend to have large central charge. More precisely, the correspondence maps a given theory of quantum gravity not just to a single VOA $V$, but rather to a whole family $\{V^N\}_{N\in\N}$ of VOAs, whose central charges are parametrized by $N$. For the purposes of the AdS/CFT correspondence, physicists are most interested in the  `large central charge limit', that is the limit of this family for $N\to\infty$. This limit corresponds to the classical limit of the quantum gravity theory; the parameter $1/N$ then plays a role similar to the Planck constant $\hbar$ in standard quantization problems. Such limits have not been defined mathematically, and much less investigated systematically. In this note, we build on previous work
\cite{Lunin:2000yv,Belin:2014fna,Haehl:2014yla,Belin:2015hwa} to define and investigate such limits for permutation orbifolds of vertex operator algebras.

The starting point of a permutation orbifold is the $N$-fold tensor product $V^{\otimes N}$ of some `seed VOA' $V$ of central charge $c$. The tensor product is again a VOA with central charge $cN$. 
Its automorphism group $Aut(V^{\otimes N})$ contains as a subgroup the symmetric group $S_N$, which acts by permuting the $N$ tensor factors. We can thus try to orbifold by any permutation group $G_N<S_N$ \cite{Klemm:1990df, Dijkgraaf:1996xw, Borisov:1997nc, Bantay:1997ek}. A family of permutation groups $\{G_N\}_{N\in\N}$ thus leads to a family of VOAs. In this note, we define the limit of this family, and investigate under what conditions it exists. We find that if it exists, the limit is naturally given by a vertex algebra $\Vi$.

To be precise, there are two related questions one can investigate. The first is the existence of a limit for the fixed point VOAs:
For any $N$, we can define the fixed point VOA $V^{G_N}$ of $V^{\otimes N}$. We can then ask the question:
given a suitable sequence $\{G_N\}_{N\in \N}$, can we define a limit $\lim_{N\to\infty} V^{G_N}$? We set up the formalism for this in section~\ref{s:oligo}. We define the notion of \emph{nested oligomorphic} sequences of permutation groups \cite{MR2581750,Belin:2014fna,Haehl:2014yla}, and show that if a certain family of group theoretic quantities converges, then the VOAs $V^{G_N}$ converge to a vertex algebra $\Vi$. The reason $\Vi$ is only a vertex algebra (VA) and not a vertex operator algebra is simply that its central charge diverges so that the norm of the conformal vector $\omega$ diverges. $\Vi$ probably still contains remnants of the conformal structure of a VOA, but it does not contain a copy of the Virasoro algebra.

The second and more complicated question is the limit of the orbifold VOAs $V^{orb(G_N)}$. For this let us assume that
$V$ is a holomorphic VOA. We can then try to extend $V^{G_N}$ to a holomorphic VOA $V^{orb(G_N)}$ by adjoining a suitable set of modules. For permutation orbifolds, this is always possible if $c$ is a multiple of 24 \cite{Evans:2018qgz}. The question is then: can we define a limit $\Vi^{orb} = \lim_{N\to\infty} V^{orb(G_N)}$? This is a much harder problem since now we also need to understand the modules of $V^{G_N}$ and their fusion rules and intertwining operators.
We leave a systematic discussion of this for future work. Instead, we present a family of $G_N$ that circumvents this problem. We show that for $G_N = \GL(N,q)$, the conformal weight of any twisted module diverges as $N\to\infty$. The limit VA $\Vi$ thus only contains the fixed point VA, and no twisted modules, so that \be
\Vi^{orb}=\Vi\ .
\ee
One reason why this is an interesting example is that we are interested in the growth of the number of states; that is, we are interested in the asymptotic behavior of $\dim \Vi_{(n)}$ (or $\dim \Vi_{(n)}^{orb}$) as $n\to \infty$. It is much easier to compute $\Vi_{(n)}$ than $\Vi_{(n)}^{orb}$, since for the former we do not need to include any twisted modules. For $\dim \Vi_{(n)}$, this has been done for various types of oligomorphic orbifolds in \cite{Belin:2014fna,Keller:2017rtk}. For $\Vi_{(n)}^{orb}$, this has only been done for the symmetric orbifold \cite{Keller:2011xi}, for which 
\be
\log \dim \Vi^{orb}_{(n)} \sim 2\pi n\ .
\ee
This exponential growth comes from the twisted modules. To obtain a different growth behavior, in particular, slower growth, orbifolds without twisted modules are needed; the $GL(N,q)$ example given above is of this type. However, it still grows exponentially fast: we find
\be
\log \dim \Vi^{orb}_{(n)} \sim  n^2\ .
\ee

\subsection{An Example: The Virasoro VOA}\label{ss:vir}
Before getting started, let us briefly discuss an example of a large $N$ limit of a family of VOAs. This discussion is somewhat informal, but will serve to illustrate the more rigorously results presented in this article. Consider the sequence $\{\textrm{Vir}_{cN}\}_{N\in\N}$ of Virasoro VOAs of central charge $cN$, where $c\in\C$ is the central charge of the first term of the sequence. We now want to define the notion of a $N\to \infty$ limit $\Vi$ of this sequence. As graded vector spaces, all terms $\textrm{Vir}_{cN}$ are of course isomorphic, with the $L_0$ graded character given by
\be\label{VirVS}
\prod_{n\geq 2}\frac1{1-q^n}\ .
\ee
The natural definition of $\Vi$ is as the graded vector space with dimensions (\ref{VirVS}). Next we want to identify sequences of vectors. To this end we define $\tilde\omega^N := N^{-1/2}\omega\in \textrm{Vir}_{cN}$ with $\omega$ the conformal vector. Its modes $\tilde L_n$ satisfy the commutation relations
\be\label{Virtildecomm}
[\tilde L_m^N,\tilde L_n^N] = \frac1{\sqrt{N}}(m-n)\tilde L^N_{m+n} + \frac c{12}m(m^2-1) \delta_{m,-n}\ .
\ee
We therefore define a limit state $\tilde \omega\in\Vi$ whose commutation relations is given by the limit of (\ref{Virtildecomm}), 
\be\label{freeVir}
[\tilde L_m,\tilde L_n] = cm(m^2-1) \delta_{m,-n}/12\ .
\ee
Clearly $\tilde\omega$ generates a Vertex Algebra $\Vi$. However, in the commutation relation (\ref{Virtildecomm}) the first term drops out in the limit due to our rescaling of the conformal vector; the resulting limit is thus no longer a VOA. Moreover, (\ref{freeVir}) shows that the structure of this VA $\Vi$ is very simple: in particular, the structure constant $f_{\tilde\omega\tilde\omega\tilde\omega}$ vanishes. There are of course structure constants of more complicated vectors that do not vanish, such as $f_{\omega\omega L_{-2}\omega}$, but these can always be reduced to the simpler structure constants using the Jacobi identity. Essentially, computing any correlation function is reduced to a combinatorial problem.

Physicists call a VA with commutation relations as in (\ref{freeVir}) \emph{free}; in their language this means that any correlation function can be computed using Wick contractions. This is actually a  desirable property for the AdS/CFT correspondence, as there are physical arguments that predict that the VA should become free. We will discuss below that this also happens for instance for the symmetric orbifold \cite{Lunin:2000yv,Belin:2015hwa}, which is the prime example of a VA that appears in the AdS/CFT correspondence.

\emph{Acknowledgments:}
We thank Jay Taylor for useful discussions. We thank Alexandre Belin and Alex Maloney for helpful discussions and comments on the draft. We thank the anonymous referees for helpful comments on the draft. TG thanks the Department of Mathematics at University of Arizona for hospitality.
The work of TG is supported by the Swiss National Science Foundation Project Grant 175494.

\section{Oligomorphic Families of Permutation Groups}\label{s:oligo}

\subsection{Unitary VOAs}

In this article, we will develop a notion of large $N$ limit for unitary vertex operator algebras. Note that we are assuming unitarity for convenience only, because it guides our construction in many examples such as the Virasoro VOA discussed in subsection \ref{ss:vir} and because it will be convenient to work with an inner product and orthonormal basis. A more complete theory of limits of VOAs using matrix elements instead of inner products will be presented in a follow-up project.

Let us first recall some definitions. For a comprehensive account of the theory of unitary VOAs, see \cite{MR3119224}. 
Let $V$ be a VOA. An anti-linear automorphism $\theta$ of $V$ is
an anti-linear isomorphism $\theta:V\to V$ such
that $\theta(\vac)=\vac, \theta(\omega)=\omega$ and
$\theta(u_nv)=\theta(u)_n\theta(v)$ for any $ u, v\in V$ and $n\in
\mathbb{Z}$.

Now let $V$ be a VOA with an anti-linear involution $\theta: V\to V$, i.e. an anti-linear automorphism of order $2$. Then $V$ is called unitary if there exists a positive definite Hermitian form $\langle, \rangle: V\times V\to \mathbb{C}$ such that the following invariant property holds: for any $a, u, v\in V$\\
 \begin{equation}
     \langle Y(e^{zL(1)}(-z^{-2})^{L(0)}a, z^{-1})u, v \rangle = \langle u, Y(\theta(a), z)v) \rangle\ ,
 \end{equation}
 where $L(n)$ is defined by $Y(\omega, z)=\sum_{n\in \Z}L(n)z^{-n-2}$.

For the remainder of this article, let $V$ be a unitary VOA of CFT type with inner product $\langle\cdot,\cdot\rangle$ and anti-linear involution $\theta$. In particular this means that
\be
V = \C \vac \oplus \bigoplus_{n\geq 1} V_{(n)}\ .
\ee
Let $\Phi:=\bigcup_{n\geq 0} \Phi_n$ be a homogeneous orthonormal basis of $V$, that is for $a,b\in \Phi$
\be
\langle a, b\rangle = \delta_{a,b}\ .
\ee
For simplicity let us also assume that $\Phi$ is real, that is $\theta(a)=a\ \forall a \in \Phi$.
Define the \emph{structure constants} 
\be
f_{abc} := \langle a, b_{\wt(c)+\wt(b)-\wt(a)-1} c \rangle\ .
\ee
We note that $f_{\vac\vac\vac}=1$ and that $f_{abc}=0$ if exactly two of the three vectors are equal to $\vac$; in physics language we say that `1-point functions vanish'.
Fixing all structure constants of the VOA is equivalent to fixing the state-field map by
\be
Y(b,z)c = \sum_{a\in \Phi}  z^{\wt(a)-\wt(b)-\wt(c)} f_{abc} a
\ee
Next let us rewrite Borcherds' identity in terms of structure constants:
\begin{prop}
Borcherds' identity is equivalent to the condition
\be\label{3ptJacobi}
\sum_{d\in\Phi} \binom{m}{j_1} f_{edc}f_{dab} = \sum_{d\in\Phi} (-1)^{j_2}\binom{n}{j_2} f_{ead}f_{dbc} - \sum_{d\in\Phi} (-1)^{j_3+n}\binom{n}{j_3} f_{ebd}f_{dac}\qquad \forall\ a,b,c,e\ ,
\ee
where $j_1=\wt(b)+\wt(a)-\wt(d)-n-1$, $j_2=\wt(c)+\wt(b)-\wt(d)-k-1$, $j_3=\wt(c)+\wt(a)-\wt(d)-m-1$. 
\end{prop}
\begin{proof}
First note that $j_i\geq 0$, implies that the sum over $d\in \Phi$ is finite.
Next, Borcherds' identity can be written as (see \eg (4.8.3) in \cite{MR1651389})
\be
\sum_{j=0}^\infty \binom{m}{j}(a_{n+j}b)_{m+k-j}c = 
\sum_{j=0}^\infty (-1)^j\binom{n}{j}a_{m+n-j}(b_{k+j}c)
- \sum_{j=0}^\infty (-1)^{j+n}\binom{n}{j}b_{n+k-j}(a_{m+j}c)
\ee
Now simply take the inner product of this expression with the basis vector $e$, and insert a complete basis ${d}$. For a given basis vector $d$, only one term in the sum over $j$ is non-vanishing: in the last sum for instance, the term with $j=\wt(c)+\wt(a)-\wt(d)-m-1$.
\end{proof}

\subsection{Oligomorphic Orbifolds}
Let $x_N$ be a sequence of positive integers and let $X_N:=\{1,2,\ldots x_N\}$ so that $|X_N| = x_N$. Further let $G_N$ a permutation group of $X_N$ and finally let $V$ be a unitary VOA of CFT type as defined above. We define the  series $a(t)= \sum_{n\geq0} a_n t^n$ with $a_n:=|\Phi_n|$, such that $a(t)= \chi_{V}(t)t^{c/24}$ is the shifted character of $V$. Note that $\Phi_0 =\{\vac \}$ and therefore $a_0=1$.
We will denote by 
\be
V^{X_N} := V^{\otimes|X_N|}
\ee
the tensor product VOA, and by
\be
V^{G_N} :=  (V^{X_N})^{G_N}
\ee
the fixed point VOA of the tensor product VOA under $G_N$, with $G_N$ acting by permuting factors. 

To give an expression for the character of $V^{G_N}$, we will use chapter 15.3 of \cite{MR1311922} here. Let $g^N$ be a function $g^N: X_N \to \Phi$. We define the weight of $g^N$ as 
\be
|g^N|:=\sum_{x\in X_N} \wt(g^N(x))\ ,
\ee
and its  support as 
\be
\supp (g^N) := \{ x\in X_N : g^N(x)\neq \vac\}\ .
\ee
$\sigma \in G_N$ acts on $g^N$ as
\be
(\sigma g^N)(x):= g^N(x \sigma^{-1})\ .
\ee
We define $b_n(G_N)$ to be the number of orbits of $G_N$ on functions of weight $n$.
The (shifted) character of $V^{G_N}$ is then given by
\be\label{fixedpointchar}
Z_{G_N}(t) = \sum_{n\geq 0}b_n(G_N) t^n\ .
\ee
This character can also written in terms of the cycle index:
\begin{defn}
	Let $G_N$ be a permutation group on $|X_N|$ elements. For $\sigma \in G_N$, denote the number of cycles of length $k$, $1\leq k\leq |X_N|$ in the cycle decomposition of $\sigma$ by $m_k(\sigma)$. Then the cycle index of $G_N$ is the following polynomial in the variables $s_1,s_2,....,s_{|X_N|}$:
	\begin{equation}
	\chi_{G_N}(s_1,s_2,s_3,....,s_{|X_N|})= \frac{1}{|G_N|}\bigg( \sum_{\sigma\in G_N} s_1^{m_1(\sigma)}s_2^{m_2(\sigma)}s_3^{m_3(\sigma)}\cdots s_{|X_N|}^{m_N(\sigma)}\bigg)
	\end{equation} 
\end{defn}
Using Thm 15.3.2 of \cite{MR1311922} we can express the fixed point character as
\be
Z_{G_N}(t) = \chi_{G_N}(a(t),a(t^2), \ldots,a(t^N))\ .
\ee
We note that this formula can also be obtained from the expression for the character of general permutation orbifolds \cite{Bantay:1997ek}. For the moment we are only interested in the fixed point algebra, and do not include the contribution of any twisted modules.

Before defining oligomorphic permutation orbifolds, let us fix some notation first:
\begin{defn}\label{def:oligo}
Let $\cK \subset X_N$.
\begin{enumerate}
\item Denote by $\sstab^{\cK} := \{\sigma \in G_N|k \sigma \in \cK, \forall k \in \cK\}$ the setwise stabilizer of $\cK$. 
\item Denote by $\pstab^{\cK} := \{\sigma \in G_N|k \sigma = k, \forall k \in \cK\}$ the pointwise stabilizer of $\cK$. 
\item Let $\Grem{\cK}^N$ be the permutation group defined by the action of 
$\sstab^{\cK}/\pstab^{\cK}$ on $\cK$. Note that $\Grem{\cK}^N$ is the restriction of $G_N$ to $\cK$ in the natural sense. \label{Gremdef}
\end{enumerate}
\end{defn}
Note that $\pstab^{\cK}$ is a normal subgroup of $\sstab^{\cK}$, so that definition (\ref{Gremdef}) makes sense.

\begin{defn}\label{conv}
Assume $|X_N|<|X_{N+1}|$.
Let the family $\{G_N\}_{N\in\N}$ satisfy the conditions:
\begin{enumerate}
    \item The numbers $b_n(G_N)$ converge for all $n$.\label{condoligo}
    \item For every finite $\cK$, there is a group $\Grem{\cK}$ such that $\Grem{\cK}^N = \Grem{\cK}$ for $N$ large enough.
    \item $\Grem{X_{N-1}}^N< G_{N-1}$ for all $N$.\label{condnest}
\end{enumerate}
We then call $G_N$ to be \emph{nested oligomorphic}.
\end{defn}

Let us give a weaker criterion for condition (\ref{condoligo}), which also motivates the name oligomorphic. Define  $f_n(G_N)$ as the number of orbits of $n$-element subsets of $X_N$ \cite{MR2581750}. 
We then have 
\begin{prop}\label{condoligo2}
Suppose $\{G_N\}_{N\in\N}$ satisfies the nesting condition (\ref{condnest}), and $f_n(G_N)$ is bounded for each $n$ as $N\to \infty$. Then  $\{G_N\}_{N\in\N}$ satisfies (\ref{condoligo}).
\end{prop}
\begin{proof}
Two different orbits of $G_{N-1}$ are in different $G_N$ orbits due to the nesting condition. It follows that
$b_n(G_N)$ grows monotonically in $N$.
To bound $b_n(G_N)$, note that for $g^N$ of fixed weight $n$, $g^N(x)=\vac$ except for at most $n$ values of $x$, and that for each of those values $g^N$ can take at most $A(n):=\sum_{j=0}^n a_j$ values. With $f_n(G_N)$ the number of orbits of $n$-element subsets of $X_N$, we have
\be\label{bnfn}
b_n(G_N) \leq A(n)^n f_n(G_N)\ .
\ee
It follows that if the $f_n(G_N)$ are bounded, then so are the $b_n(G_N)$, which means that they converge.
\end{proof}

It is, for example, easy to see that for $|X_N| = N$ the family $G_N = S_N$ is nested oligomorphic, while the family $G_N = \Z_N$ is not:
For $G_N = S_N$, we find that $f_n(S_N) = 1$ and $\Grem{\cK}^N = S_{|\cK|}$. For $G_N = \Z_N$ on the other hand, we find that $b_n(\Z_N)$ diverges for all $n > 3$. Consider for example orbits of weight $4$ states with exactly two non-vacuum factors equal to the conformal vector. We find that there are at least $N/2-1$ such orbits with representatives given by states with one conformal vector in the first factor and the other in the $i$-th factor for $i = 2, \ldots, N/2$. Other examples of oligomorphic groups include direct products and wreath products of symmetric groups \cite{Belin:2015hwa,Keller:2017rtk}. 

For oligomorphic permutation orbifolds the fixed point characters (\ref{fixedpointchar}) converge in the limit $N\to\infty$, at least as formal power series in $t$. The limit of these functions and the asymptotic growth behavior of the limit coefficients $b_n$ as $n$ goes to infinity were discussed for various oligomorphic orbifolds in \cite{Belin:2014fna,Belin:2015hwa,Keller:2017rtk}.
Part of the motivation for this work is to understand the growth of the coefficients $b_n$ better, and construct limit VAs with new growth behavior. We return to this question at the end of section~\ref{s:GLNq}.

\subsection{Space of states}
Take $n$ to be fixed, and
let $\cB_n^N$ be the set of orbits on functions on $X_N$ of weight $n$. We will now the use the functions $g^N$ introduced in section \ref{s:oligo} to define vectors in the tensor product VA. Note that $|\supp(g^N)| \leq n$. For any $g^N: X_N \to \Phi$ with weight $|g^N|=n$, define the following vector in the tensor product VOA $(V^{\otimes |X_N|})_{(n)}$,

\be
\hat \phi_{g^N} = \bigotimes_{x\in  X_N} g^N(x)\ ,
\ee
and the following vector in the fixed point VOA $(V^{G_N})_{(n)}$,
\be\label{fixedptstate}
\phi_{g^N} = A(g^N,N)^{-1/2} \sum_{\sigma\in G_N} \hat\phi_{\sigma g^N} \in (V^{G_N})_{(n)}
\ee
where $A(g^N,N)$ is a normalization factor 
that ensures that $\langle\phi_{g^N},\phi_{g^N}\rangle=1$;
we will give an explicit expression below. 
Note that two $\phi_{g^N}$ are automatically orthogonal if $g_1^N$ and $g_2^N$ are in different orbits: the inner product between $\hat\phi^N_{g_1}$ and $\hat\phi^N_{g_2}$ is only non-vanishing if $g_1=g_2$, which means that $\phi^N_{g_1}$ and $\phi^N_{g_2}$ only have non-vanishing terms if $g_1$ and $g_2$ are in the same orbit.

Now let $\{G_N\}$ be nested oligomorphic. We have:
\begin{prop}\label{grep}
For each $n$, we can find $N_n$ such that for any $N\geq N_n$, there exists a representative $g^N$ of $\cB_n^N$ such that $\supp(g^N) \subset X_{N_n}$ and $g^N|_{X_{N_n}}= g^{N_n}$.
\end{prop}
\begin{proof}
Pick $N_n$ such that $|\cB_n^{N_n}|$ has converged already. Pick representatives $g^{N_n}$ of $\cB_n^{N_n}$, and extend them to $g^N$ by $g^{N}(j)=\vac$ if $j\notin X_{N_n}$. These are in different orbits due to the nesting condition. Since $|\cB_n^{N_n}|$ has converged, we find such a representative for all orbits in $\cB_n^N$.
\end{proof}
Let us denote $\cB_n := \cB_n^{N_n}$, $X_n := X_{N_n}$ and $g:=g^{N_n}$.
For $N\geq N_n$, using  (\ref{fixedptstate}) we can then define the embedding
\be
\iota_N : \cB_n \to (V^{G_N})_{(n)}\ ,\qquad g \mapsto \phi_{g^N}\ ,
\ee
where $g^N$ is from Prop~\ref{grep}.
Now let $\cK:=\supp(g^N)$. For the normalization factor $A(g^N,N)$ we obtain 
\begin{multline}\label{normalization}
A(g^N,N) = \sum_{\tau, \sigma\in G_N} \langle \hat\phi_{\tau g^N},\hat\phi_{\sigma g^N }\rangle= |G_N| \sum_{\sigma\in G_N} \langle\hat\phi_{g^N},\hat\phi_{\sigma g^N }\rangle = |G_N| \sum_{\sigma\in \sstab^{\cK}} \langle\hat\phi_{g^N},\hat\phi_{\sigma g^N }\rangle\\
= |G_N||\pstab^{\cK}| \sum_{\sigma\in \Grem{\cK}^N}
\langle\hat\phi_{g^N},\hat\phi_{\sigma g^N}\rangle
=: A(g)|G_N||\pstab^{\cK}|,
\end{multline}
where in the first step we use that $\langle\hat\phi_{\tau g^N},\hat\phi_{\tau g^N }\rangle = \langle\hat\phi_{g^N},\hat\phi_{g^N }\rangle$, in the second step that $\langle\hat\phi_{g^N},\hat\phi_{\sigma g^N }\rangle$ vanishes unless $\supp(g^N) = \supp(\sigma g^N)$ and in the last step that $\langle\hat\phi_{g^N},\hat\phi_{\sigma \tau g^N }\rangle = \langle\hat\phi_{g^N},\hat\phi_{\sigma g^N }\rangle$ if $\tau g = g$. Note that for $N$ large enough $A(g)$ is indeed independent of $N$: this follows from the fact that $\Grem{\cK}^N$ converges to $\Grem{\cK}$, and all factors in $\cK^c$, the complement of $\cK$ in $X_N$, 
in the inner product give contribution 1.

From the above it follows:
\begin{prop}
For $N\geq N_n$, $\iota_N(\cB_n)$ forms an orthonormal basis for $(V^{G_N})_{(n)}$.
\end{prop}

\subsection{Structure constants}
Let us now work out explicit expressions for the structure constants of the fixed point algebra.
For $g_i \in \cup_n \cB_n$, define
\be\label{Cg1g2g3}
C(g_1,g_2,g_3):= \prod_{x=1}^{|X_N|}f_{g_1(x)g_2(x)g_3(x)}
\ee
as the structure constant of the tensor VOA. Note that since $f_{\vac\vac\vac}=1$, the choice of $N$ does not matter, as long as it is large enough. We have therefore suppressed the index $N$ in $g_i^N$.
Our goal is now to compute the structure constant of the fixed point VOA, 
\be\label{3ptsum}
f^N_{\phi^N_{g_1}\phi^N_{g_2}\phi^N_{g_3}}
=\bigg(\prod_{i=1}^3 A(g_i,N)^{-1/2}\bigg) \sum_{\sigma \in G_N^{\times 3}} C(\sigma_1 g_1, \sigma_2 g_2, \sigma_3 g_3)\ ,
\ee
where $\sigma = (\sigma_1, \sigma_2, \sigma_3)$.
To do this, we establish the following expression for the structure constants. Denote the support of $g_i$, \ie the set for which $g_i(j)\neq \vac$ by $\cK_i$.
\begin{thm}\label{thm:3pt}
Let $G_N$ be nested oligomorphic.
We then have 
\be\label{eq:oligo}
f^N_{\phi^N_{g_1}\phi^N_{g_2}\phi^N_{g_3}}=
\bigg(\prod_{i=1}^3 A(g_i)^{-1/2}\bigg)\sum_{[\kappa]\in \cS}  M(\kappa,N) \sum_{[\sigma] \in \bigtimes_i \Grem{\cK_i}} C(\kappa_1\sigma_1 g_1,\kappa_2\sigma_2 g_2, \kappa_3\sigma_3 g_3)\ ,
\ee
where $\sigma = (\sigma_1, \sigma_2, \sigma_3) \in G_N^{\times 3}$, $\kappa = (\kappa_1, \kappa_2, \kappa_3) \in G_N^{\times 3}$ and
\be
\cS= G_N^{\diag} \backslash G_N\times G_N \times G_N/ \sstab^{\cK_1}\times \sstab^{\cK_2} \times \sstab^{\cK_3}
\ee
as a set, where $G_N^{\diag}$ is the diagonal subgroup of $G_N\times G_N\times G_N$, and
\be\label{MGLmq}
M(\kappa,N)  =
\left(\frac{|\pstab^{\cK_1}||\pstab^{\cK_2}||\pstab^{\cK_3}|}{|G_N||\pstab^{\kappa_1\cK_1\cup\kappa_2\cK_2\cup\kappa_3\cK_3}|^2}\right)^{1/2}
\ee
\end{thm}
Before we prove theorem~\ref{thm:3pt}, let us make some remarks on the intuition behind it. First note that the sum over $[\sigma]$ is indeed well-defined, that is independent of the choice of the representative of $[\kappa]$: $C(\kappa_1 g_1, \kappa_2 g_2, \kappa_3 g_3)$ is invariant under the action of $G^{\diag}$ since it is simply a product of structure constants; it is invariant under $\pstab^{\cK_i}$, since this only permutes factors with vacuum vectors; and acting with $\sstab^{\cK_i}$ only permutes the elements of $\Grem{\cK_i}$, that is the terms in the sum over $[\sigma]$. Similarly $|\pstab^{\kappa_1\cK_1\cup\kappa_2\cK_2\cup\kappa_3\cK_3}|$ only depends on the conjugacy class $[\kappa]$ since it is invariant under $\sstab^{\cK_i}$, which leaves $\cK_i$ invariant, and $G_N^{\diag}$, which simply gives an overall conjugation of the stabilizer, leaving its order invariant.

Next we note that for $N$ large enough, the only component in (\ref{eq:oligo}) that depends on $N$ is $M(\kappa,N)$. The idea is that because $G_N$ is nested oligomorphic, by choosing an appropriate element of $G_N^{\diag}$, we can always find a representative such that all three $\kappa_i\sigma_ig_i$ have support in the first say $n$ factors, where $n$ is independent of $N$ for large enough $N$. In the pictorial notation of \cite{Belin:2015hwa}, such a $\kappa$ might for instance give the following:
\begin{eqnarray}\label{pictorial}
\kappa_1\sigma_1 g_1:&&\overbrace{
\bb\bb\bb\bb\bb\bb\bb\bb\bb\bb\bb\bb\bb\bb
\wb\wb\wb\wb
\wb\cdots\wb}^{X_N}\nonumber\\
\kappa_2\sigma_2 g_2:&&
\bb\bb\bb\bb\bb\bb\bb\bb\bb\bb
\wb\wb\wb\wb
\bb\bb\bb\bb
\wb\cdots\wb \\
\kappa_3\sigma_3 g_3:&&
\bb\bb\bb\bb\bb
\wb\wb\wb\wb\wb
\bb\bb\bb\bb\bb\bb\bb\bb
\wb\cdots\wb\nonumber
\end{eqnarray}
Here black bullets indicate a non-vacuum vector in the corresponding factor, and white bullets vacuum vectors. $C(\kappa_1\sigma_1 g_1,\kappa_2\sigma_2 g_2, \kappa_3\sigma_3 g_3)$ is then independent of the number of trailing white bullets, since the structure constant $f_{\vac\vac\vac}$ is equal to 1. From the point of view of the pictorial representation, the sum over $[\kappa]$ parametrizes the different configurations \`a la (\ref{pictorial}); $M(\kappa,N)$ gives their multiplicity; and the sum over the $\Grem{\cK_i}$ gives the sum over the different permutations among the black bullets. 

Let us now prove theorem~\ref{thm:3pt} by making these remarks more formal.

\begin{proof}

To prove the theorem, first consider the quotient set
\be\label{cSt}
\cS_1 = G_N\times G_N \times G_N  / (\pstab^{\cK_1}\times \pstab^{\cK_2} \times \pstab^{\cK_3})\ .
\ee
Note that this is not a group since the group with which we quotient is not normal. For our purposes it is enough that the quotient exists as a set.
The diagonal subgroup $G_N^{\diag}$ acts on $\cS_1$ as
\be
\tau^{\diag} \cdot (\sigma_1\pstab^{\cK_1},\sigma_2\pstab^{\cK_2},\sigma_3\pstab^{\cK_3}) = (\tau\sigma_1\pstab^{\cK_1},\tau\sigma_2\pstab^{\cK_2},\tau\sigma_3\pstab^{\cK_3})\ .
\ee
Alternatively, since we are only interested in the quotient as a set, we can use the orbit-stabilizer theorem to describe $\cS_1$ as the set of images of $\cK_1\times\cK_2\times\cK_3$ under $G_N\times G_N \times G_N$,
\be\label{S1orbstab}
\cS_1 \cong (G_N\times G_N \times G_N)\cdot (\cK_1\times\cK_2\times\cK_3)\ .
\ee
Diagrams like (\ref{pictorial}) then depict an element $\alpha\in\cS_1$, with black bullets indicating elements of $\sigma_i\cK_i$, and white bullets indicating elements in its complement.
$G_N^{\diag}$ then acts in the obvious way on the set $\cS_1$. This allows us to define the quotient set
\be
\cS_2 =  G_N^{\diag}\backslash \cS_1 \ .
\ee
If $G_N$ is nested oligomorphic, then $\cS_2$ converges as $N\to\infty$: that is, $|\cS_2|$ converges, and for every equivalence class we can find an $N$-independent representative $\beta\in\cS_2$. To see this we use the description (\ref{S1orbstab}) of $\cS_1$. Since $G_N$ is nested oligomorphic, we can use the same argument as proposition~\ref{grep}: for large enough $N$ we find $\tau\in G_N$ such that $\tau (\sigma_1\cK_1\cup \sigma_2\cK_2\cup\sigma_3\cK_3) \subset X_{\bar N}$ where $\bar N$ is independent of $N$. This gives the $N$-independent representative $\beta$. In (\ref{pictorial}) for instance $\beta$ was chosen such that its support is in the first 18 elements of $X_N$.

To compute the length of an orbit $[\beta]\in\cS_2$, we apply the orbit-stabilizer theorem to the action of $G_N^\diag$ on $\cS_1$. For a given $[\alpha]\in\cS_1$, the stabilizer is given by
\be
\pstab^{\alpha_1\cK_1\cup \alpha_2\cK_2\cup\alpha_3\cK_3}\ ,
\ee
so that by the orbit-stabilizer theorem the orbit has length 
\be
|[\beta]| = |G_N|/|\pstab^{\alpha_1\cK_1\cup \alpha_2\cK_2\cup\alpha_3\cK_3}|
= |G_N|/|\pstab^{\beta_1\cK_1\cup \beta_2\cK_2\cup\beta_3\cK_3}|\ .
\ee
Next note that $C(g_1,g_2,g_3)$ is invariant under $\pstab^{\cK_i}$, which only permutes identical vacuum factors in the three individual states, and $G_N^{diag}$, which only permutes overall factors in (\ref{Cg1g2g3}). This means we can write (\ref{3ptsum}) as
\be\label{S2sum}
\sum_{\sigma \in G_N^{\times 3}} C(\sigma_1 g_1, \sigma_2 g_2, \sigma_3 g_3) = \sum_{[\beta]\in\cS_2}\bar M(\beta,N) C(\beta_1 g_1,\beta_2 g_2,\beta_3 g_3) 
\ee
where $\bar M(\beta;N)$ is the length of the equivalence class, given by
\be
\bar M(\beta;N)= \frac{|G_N||\pstab^{\cK_1}||\pstab^{\cK_2}||\pstab^{\cK_3}|}{|\pstab^{\beta_1\cK_1\cup \beta_2\cK_2\cup\beta_3\cK_3}|}\ .
\ee
Again by oligomorphicity we can choose a representative of $[\beta^2]$ that is $N$-independent. 
Next we consider $\Grem{\cK}^N = \sstab^\cK/\pstab^\cK$. Note that here we want to use the fact this is a group, and not just a quotient set. In fact this group has a well-defined right action on $\cS_2$, so that we can define the quotient set 
\be
\cS := \cS_2 / \Grem{\cK_1}^N\times \Grem{\cK_2}^N \times \Grem{\cK_3}^N\ .
\ee
Since as sets
\be
\big(\frac{G_N}{\hat G_N^{\cK}}\big)\big/ \Grem{\cK} \cong G_N/G_N^\cK\ ,
\ee
we can also write $\cS$ as
\be
\cS = G_N^{\diag}\backslash G_N\times G_N \times G_N/(\sstab^{\cK_1}\times \sstab^{\cK_2} \times \sstab^{\cK_3})\ .
\ee
Finally we take $N$ large enough such that by property (\ref{Gremdef}) of definition~\ref{def:oligo} $G(\cK_i)^N=G(\cK_i)$.
This allows us to write the sum over $\cS_2$ in (\ref{S2sum}) as a sum over $\cS$ times a sum over the $G(\cK_i)$.
Including the normalization (\ref{normalization}), we can thus write (\ref{3ptsum}) as
\be
f^N_{\phi^N_{g_1}\phi^N_{g_2}\phi^N_{g_3}}=
\sum_{[\kappa]\in \cS}  M(\kappa,N) \sum_{\sigma \in \bigtimes_i \Grem{\cK_i}} A(g_i)^{-1/2}C(\kappa_1\sigma_1 g_1,\kappa_2\sigma_2 g_2, \kappa_3\sigma_3 g_3)\ ,
\ee
where
\be
M(\kappa,N) := \prod_i (|G_N||\pstab^{\cK_i}|)^{-1/2}\bar M(\kappa,N) =
\left(\frac{|\pstab^{\cK_1}||\pstab^{\cK_2}||\pstab^{\cK_3}|}{|G_N||\pstab^{\kappa_1\cK_1\cup\kappa_2\cK_2\cup\kappa_3\cK_3}|}\right)^{1/2}
\ee
\end{proof}

Using the above results, we can now define a limit VA in the following way:

\begin{thm}\label{largeNlimit}
Assume $V$ is a unitary VOA of CFT type, $G_N$ is nested oligomorphic, and the limits
\be
 f_{g_1 g_2 g_3 }:= \lim_{N\to\infty}f^N_{g_1 g_2 g_3} 
\ee
exists for all $g_1,g_2,g_3$, where $f^N_{g_1 g_2 g_3}$ denotes the structure constant of the VOA $V^{G_N}$.
Then define
$\Vi = \bigoplus_{n=0}^\infty V_{(n)}$ with 
\be
V_{(n)} :=\bigoplus_{g\in\cB_n} \C \phi_g\ 
\ee
and 
\be
Y(\phi_{g_2},z)\phi_{g_3} = \sum_{g_1} f_{g_1 g_2 g_3} \phi_{g_1} z^{|g_1|-|g_2|-|g_3|}\ .
\ee
Then  $(V,Y)$ defines a vertex algebra.
\end{thm}
\begin{proof}
To prove that $(V,Y)$ defines a VA, we establish that the structure constants $f_{g_1g_2g_3}$ satisfy the Jacobi identity, and that $Y$ satisfies the creation property $Y(a,z)\vac = a + O(z)$ and the vacuum property $Y(\vac,z) = \mathrm{id}$. These three then imply the more commonly used axioms of locality and translation \cite{MR2023933}: the translation $T$ for instance is defined as $T: a\to a_{-2}\vac$ and automatically satisfies the translation axiom because of the Jacobi identity. To show the creation property, note that for finite $N$ we have $f^N_{\phi^N_{g_1}\phi^N_{g_2}\phi^N_{\vac}}=\delta_{\phi^N_{g_1},\phi^N_{g_2}}$ for $|g_1|=|g_2|$. This also holds for the limit, so that indeed $Y(a,z)\vac = a + O(z)$. Similarly, to show the vacuum property, note that for all finite $N$ we have $f^N_{\phi^N_{g_1}\phi^N_{\vac}\phi^N_{g_2}}=\delta_{\phi^N_{g_1},\phi^N_{g_2}}$ for $|g_1|=|g_2|$. Since the $f^N$ satisfy (\ref{3ptJacobi}) for all $N$, and the sum is over a finite and $N$-independent number of terms, we can exchange sum and limit. It then follows that the limit also satisfies (\ref{3ptJacobi}).
\end{proof}

To establish the existence of the $N\to\infty$ limit, it is thus enough to ensure that $M(\kappa,N)$ converges:
\begin{cor}
If $\{G_N\}$ is nested oligomorphic and the $ M(\kappa,N)$ converge for all $\kappa$, then the fixed point VOA $V^{G_N}$ has a large $N$ VA limit.
\end{cor}

\subsection{An example: \texorpdfstring{$S_N$}{SN}}\label{ss:SN}
Let us illustrate the above by briefly discussing the best known example of a permutation orbifold VOA with a large $N$ limit, namely $G_N=S_N$. For this we rephrase the result of \cite{Belin:2015hwa} in the language above. $S_N$ is nested oligomorphic: 
We have $f_K(S_N)=1$ and with $K = |\cK|$
\be
\pstab^{\cK}= S_{N-K} \ ,\qquad \sstab^{\cK}= S_{N-K}\times S_K\ ,
\ee
so that 
\be
\Grem{\cK}^N = S_K\ .
\ee
In particular also $S_N|_{X_{N-1}} = S_{N-1}$. To show that the limit exist, in view of theorem~\ref{largeNlimit} it is enough to show that $ M(\kappa,N)$ converges.

To evaluate $ M(\kappa,N)$, we need to find $|{\kappa_1\cK_1\cup\kappa_2\cK_2\cup\kappa_3\cK_3}|$. Using the same notation as above, a configuration $[\kappa]$ looks something like this:
\begin{eqnarray}
\kappa_1g_1:&&\overbrace{\underbrace{\bb\bb\bb\bb\bb\bb\bb\bb\bb\bb\bb\bb\bb\bb}_{K_1}\underbrace{\wb\wb\wb\wb}_{K_2-J}\wb\cdots\wb}^N\nonumber\\
\kappa_2g_2:&&\underbrace{\bb\bb\bb\bb\bb\bb\bb\bb\bb\bb}_J\wb\wb\wb\wb\underbrace{\bb\bb\bb\bb}_{K_2-J}\wb\cdots\wb\label{SNkappa}\\
\kappa_3g_3:&&\underbrace{\bb\bb\bb\bb\bb}_{n_3}\wb\wb\wb\wb\wb\underbrace{\bb\bb\bb\bb\bb\bb\bb\bb}_{K_1+K_2-2J}\wb\cdots\wb
\nonumber
\end{eqnarray}
Here we have used the diagonal action $G_N^{\diag}$ to pick a representative  of $[\kappa]$ that moves the support of all three $g_i$ to the very left. Note that this configuration has either 3,2 or 0 black bullets in each column. We claim that this is in fact the most general configuration that makes a contribution:
Note that $C(\kappa_1 g_1,\kappa_2g_2,\kappa_3g_3)$ vanishes if there is $x\in X_N$ such that $\kappa_i g_i(x)=\vac$ for exactly two of the three indices $i$. We can thus restrict to  configurations $\kappa$ where for all $x\in {\kappa_1\cK_1\cup\kappa_2\cK_2\cup\kappa_3\cK_3}$, $\kappa_i g_i(x)=\vac$ for either one or zero values of $i$. 

For fixed $\kappa$, define  $n_3(\kappa):=|{\kappa_1\cK_1\cap\kappa_2\cK_2\cap\kappa_3\cK_3}|$, that is the number of $x$ such that $\kappa_i g_i(x)\neq\vac$ for $i=1,2,3$. It turns out that because of this, $n_3(\kappa)$ completely determines $|{\kappa_1\cK_1\cup\kappa_2\cK_2\cup\kappa_3\cK_3}|$ \cite{Belin:2015hwa}:
For any subsets $\cK_i$ we have
\begin{align*}
    |{\cK_1\cup\cK_2\cup\cK_3}| & = |\cK_1| + |\cK_2| - |\cK_1 \cap \cK_2| + |\cK_3| - |(\cK_1\cup\cK_2)\cap \cK_3| \\
    & = |\cK_1| + |\cK_2| - |\cK_1 \cap \cK_2| + |\cK_3| - |\cK_1\cap \cK_3| - |\cK_2\cap \cK_3| + |\cK_1\cap \cK_2\cap\cK_3|.
\end{align*}
Due to the remark above about $1$-point functions vanishing, a configuration gives vanishing contribution unless the $\cK_i$ satisfy
\begin{align*}
    |\cK_1| & = |\cK_1 \cap \cK_2| + |\cK_1\cap \cK_3| - |\cK_1\cap \cK_2\cap\cK_3| \\
    |\cK_2| & = |\cK_2 \cap \cK_3| + |\cK_1\cap \cK_2| - |\cK_1\cap \cK_2\cap\cK_3| \\
    |\cK_3| & = |\cK_3 \cap \cK_1| + |\cK_2\cap \cK_3| - |\cK_1\cap \cK_2\cap\cK_3|.
\end{align*}
Adding up gives
\be
|\cK_1 \cap \cK_2| + |\cK_2 \cap \cK_3| + |\cK_3 \cap \cK_1| = \frac{1}{2}(|\cK_1| + |\cK_2| + |\cK_3| + 3|\cK_1\cap \cK_2\cap\cK_3|)\ ,
\ee
which after substituting gives 
\be\label{eq:dim}
|{\kappa_1\cK_1\cup\kappa_2\cK_2\cup\kappa_3\cK_3}|=
\frac12(K_1+K_2+K_3-n_3(\kappa))\ .
\ee 
This formula can of course also be read off in a more graphical manner from (\ref{SNkappa}), where $J:=|\kappa_1\cK_1 \cap \kappa_2\cK_2|$. 
We thus get
\be
 M(\kappa,N) =\left(\frac{(N-K_1)!(N-K_2)!(N-K_3)!}{N!(N-\frac12(K_1+K_2+K_3-n_3(\kappa)))!^2}\right)^{1/2}\ .
\ee
This establishes that $M(\kappa,N)$ only depends on $n_3(\kappa)$, and no other properties of the configuration $\kappa$.
Using Stirling's approximation it follows that for $n_3>0$
\be
 M = O(N^{-n_3(\kappa)/2})
\ee
for $N\to\infty$, and for $n_3(\kappa)=0$ 
\be
 M \to 1\ . 
\ee
This shows that the family $V^{S_N}$ has a large $N$ limit. 

In fact, this establishes an even stronger result: In the large $N$ limit, only configurations with $n_3(\kappa)=0$ contribute. This implies that $S_N$ orbifolds become free VAs in the large $N$ limit in the sense discussed in the introduction. More precisely, let us call fields $\phi_g$ with only one non-vacuum factor, that is with $|\supp(g)|=1$, \emph{single trace} fields. In the large $N$ limit all structure constants of three single trace fields vanish because they have $n_3(\kappa)=1$. The only non-vanishing structure constant is thus the one involving one vacuum, which is simply the norm of the single trace field. These single trace fields generate the VA, and any structure constant can thus be computed in a combinatorial way by using the Jacobi identity. See \cite{Belin:2015hwa} for a more detailed argument.

\section{The limit of \texorpdfstring{$\GL(N,q)$}{GL(N,q)} orbifolds}\label{s:GLNq}
Let us now discuss another class of permutation orbifolds, based on the finite groups $\GL(N,q)$.
Let $q$ be such that $\F_q$ is a field. $\GL(N,q)$ is the general linear group of the finite vector space $\F_q^N$, which has $|\F_q^N|= q^N$ elements. Since any element of $\GL(N,q)$ is a bijective map, the action of $\GL(N,q)$ on $\F_q^N$ defines a permutation group acting on $q^N$ elements. Now let $X_N = \{1,2,\ldots q^N\}\simeq\F_q^N$, where the points in $\F_q^N$ are ordered lexicographically in their coordinates such that we associate to every $v=(v_1, \ldots, v_N) \in \F_q^N$ the element $\sum_{i=1}^N v_i q^{i-1} \in X_N$. In particular, this gives an embedding of the $\F_q^N$ in $\F_q^{N+1}$, and hence a natural embedding of $X_N$ in $X_{N+1}$. Then $\GL(N,q)$ acts on $V^{\F_q^N}$ as a permutation group.
We want to show that $V^{orb(\GL(N,q))}$ has a large $N$ limit. 

\begin{lem}\label{lem:fngrowth}
Let $f_n$ be the number of orbits of $n$-element subsets. We then have:
\be
\log f_n(\GL(N,q)) = \frac{n^2}4 \log q + O(n\log n)\ .
\ee
\end{lem}
\begin{proof}
	Take a set $\cK$ of $n$ vectors in $\F_q^N$. Pick $N\geq n$.
	Denote by $\vK$ the subspace spanned by $\cK$, and let $\dK:=\dim \vK$. We can always find  $\sigma\in \GL(N,q)$ which maps the $\dK$ independent vectors in $\vK$ to the unit vectors $e_1,e_2\ldots e_\dK$. The remaining $n-\dK$ linearly dependent vectors then live in the subspace spanned by the $e_i$, which contains $q^\dK$ vectors. It follows that there are at most $q^{\dK(n-\dK)}$ different orbits. The total number of different orbits is thus bounded by
	\be
	f_n \leq \sum_{\dK=0}^n q^{{\dK(n-\dK)}} \leq (n+1) q^{n^2/4}\ .
	\ee
On the other hand we can bound $f_n$ from below by picking $\dK=n/2$ linearly independent vectors, and choosing the remaining $n/2$ dependent vectors from their span, giving
	\be
	f_n \geq \binom{q^{n/2}}{n/2} \sim q^{n^2/4}n^{-n}\ .
	\ee
\end{proof}

\begin{prop}
	The family $\{\GL(N,q)\}_N$ is nested oligomorphic. Moreover in the limit $N\to\infty$ the $b_n(G_N)$ grow like
	\be\label{bnasymp}
	\log b_n = \alpha n^2 + o(n^2)\ 
	\ee
	for some constant $\alpha$. If the seed $V$ has $\dim V_{(1)} >0$, we have $\alpha = \log q /4$.
\end{prop}
\begin{proof}
With the same notation as above, we have
\be
G_N^{\cK} \simeq \GL(\dK,q)\times \GL(N-\dK,q)\ 
\ee
and 
\be\label{GLNLhat}
\hat G_N^{\cK} \simeq \GL(N-\dK,q)\ 
\ee
such that $\Grem{\cK}= \GL(\dK,q)$.
	Clearly we also have $\GL(N,q)|_{X_{N-1}} = \GL(N-1,q)$, so that using proposition~\ref{condoligo2}, we use lemma~\ref{lem:fngrowth} to establish that the $f_n(\GL(N,q))$ are bounded.
	
	To show  (\ref{bnasymp}), note that
	in (\ref{bnfn}) we can bound $A(n) < e^{C\sqrt{n}}$, which gives an upper bound on $b_n$ of the form (\ref{bnasymp}). On the other hand we can use the lower bound $f_n$ to establish the lower bound on $b_n$. If $\dim V_{(1)}>0$, then we take the function $g(x)=a\in V_{(1)}\forall x \in \supp(g)$, and else 
	 $g(x)=\omega\in V_{(2)}\forall x \in \supp(g)$, with $\omega$ the conformal vector. This vector has weight $n$ and $2n$ respectively, which leads to the bounds $b_n \geq f_n$ and $b_{2n}\geq f_n$.
\end{proof}
We can now use nested oligomorphicity to establish that a $N\to\infty$ limit exists:
\begin{prop}
The fixed point VOA for the orbifold group $\GL(N,q)$ has a large $N$ limit.
\end{prop}
\begin{proof}
To evaluate $ M(\kappa, N)$, we use (\ref{GLNLhat}) together with the asymptotic expression 
\be
|\GL(N,q)| = q^{N^2}\phi(q^{-1})(1+ O(q^{-N}))\ ,
\ee
where $\phi(q)$ is Euler's function
\[
\phi(q) = \prod_{n=1}^{\infty}(1-q^n).
\]
To obtain $|\pstab^{\kappa_1\cK_1\cup\kappa_2\cK_2\cup\kappa_3\cK_3}| = |\GL(N - \dim (\vK_1+\vK_2+\vK_3),q)|$, we repeat an argument similar to  section~\ref{ss:SN}. Namely we have the inequality
\begin{multline}
    \dim (\vK_1+\vK_2+\vK_3)  = \dK_1+\dK_2 - \dim(\vK_1 \cap \vK_2) + \dK_3 - \dim((\vK_1+\vK_2)\cap \vK_3) \\
     \leq \dK_1+\dK_2 + \dK_3- \dim(\vK_1 \cap \vK_2) - \dim(\vK_2 \cap \vK_3)-\dim(\vK_1 \cap \vK_3) + 
     \dim(\vK_1 \cap \vK_2 \cap \vK_3)\ .
\end{multline}
Next, because again $1$-point functions vanish, it follows that $\vK_1 \subset \vK_2 +\vK_3$. This means that we have
\be
\dK_1  =  \dim(\vK_1 \cap \vK_2) + \dim(\vK_1 \cap \vK_3) - 
     \dim(\vK_1 \cap \vK_2 \cap \vK_3)\ ,
\ee
and similar for $\dK_2$ and $\dK_3$. Adding everything together as above gives
\be
\dim (\vK_1+\vK_2+\vK_3) \leq \frac12(\dK_1+\dK_2+\dK_3-\dim(\vK_1 \cap \vK_2 \cap \vK_3))\ .
\ee
Defining $n_3(\kappa) := \dim \langle \kappa_1\cK_1\cap\kappa_2\cK_2\cap\kappa_3\cK_3\rangle $,
we have
\be
\dim \langle \kappa_1\cK_1\cup\kappa_2\cK_2\cup\kappa_3\cK_3\rangle\leq 
\frac12(K_1+K_2+K_3-n_3(\kappa))\ .
\ee 
For $n_3>0$ this gives the asymptotic expression
\be
 M = O(q^{-Nn_3})\ ,
\ee
and for $n_3=0$ and $\dim \langle \kappa_1\cK_1\cup\kappa_2\cK_2\cup\kappa_3\cK_3\rangle=\frac12(K_1+K_2+K_3)$ we have
\be
M \to 1\ .
\ee
This again shows that three point functions converge as $N\to \infty$, and also that $\Vi$ is free.
\end{proof}

The theory of extensions of rational VOAs was described in \cite{2015CMaPh.337.1143H}. Specifically, a rational VOA can be extended by adjoining a suitable set of its modules. Here we are mainly interested in the question if the fixed point VOA $V^{G_N}$ of a holomorphic VOA, which is still rational but no longer holomorphic, can be extended back to a holomorphic VOA $V^{orb(G_N)}$.  \cite{Evans:2018qgz} establish necessary and sufficient conditions for the existence of such holomorphic extensions of orbifolds of holomorphic VOAs. In particular, by their results there are holomorphic extensions $V^{orb(\GL(N,q))}$ for all $N$. We want to establish that these have the following limit:
\begin{thm}\label{Vorblimit}
Let $V$ be a unitary holomorphic VOA of CFT type with $c\in24\N$. Then the family of holomorphic orbifolds $V^{orb(\GL(N,q))}$ has a large $N$ VA limit $\Vi^{orb}$ which is equal to the limit $\Vi$ of the fixed point VOAs $V^{\GL(N,q)}$.
\end{thm}
 To prove this, we will establish that the conformal weight of a vector in $V^{orb(\GL(N,q))}$ remains finite in the limit $N\to\infty$ if and only if this vector comes from $V^{\GL(N,q)}$ itself rather than from a $V^{\GL(N,q)}$-module. Vectors of diverging conformal weight, that is all contributions from $V^{\GL(N,q)}$-modules, then do not contribute in the large $N$ limit, so that the limites of the orbifold $V^{\GL(N,q)}$ and its holomorphic extension $V^{orb(\GL(N,q))}$ coincide. 
 
 It was shown in \cite{MR3715704} that any irreducible $V^{\GL(N,q)}$-module arises as a $V^{\GL(N,q)}$-submodule of a irreducible $g$-twisted $V^{\F_q^N}$-module for some $g \in \GL(N,q)$. We will make use of this fact and show that the conformal weights of all relevant twisted $V^{\F_q^N}$-modules diverge.
 To this end we establish the following results:

\begin{lem}{\label{lem:conf_wgt}}
    Let $g \in S_N$ be a permutation of order $n$ and cycle type $C_g = \prod_{t\mid n}t^{b_t}$ acting on the holomorphic VOA $V^{\otimes N}$ of central charge $Nc$. Then the conformal weight of the unique irreducible $g$-twisted module is given by
    \begin{equation}
        \rho_g = \frac{c}{24}\sum_{t \mid n}b_t\big(t-\frac{1}{t}\big).
    \end{equation}
\end{lem}
\begin{proof}
This follows from taking the $S$ transform of the twining character with $g$ inserted, and reading off the conformal weight from its expansion.
\end{proof}
\begin{cor}{\label{cor:rho_min}}
 Let $g$ have $N-r$ $1$-cycles. Then the conformal weight of the irreducible $g$-twisted module is bounded from below by
 \begin{equation}
     \rho_g > \frac{rc}{32}.
 \end{equation}
\end{cor}
\begin{proof}
Lemma \ref{lem:conf_wgt} shows that for every $t$-cycle of $g$ there is a contribution of $\frac{c}{24}\big(t-\frac{1}{t}\big)$ to the conformal weight. We thus find that the contribution per element in a $t$-cycle is given by
\begin{equation}
    \rho_t = \frac{c}{24}\frac{t^2-1}{t^2} > \frac{c}{32}.
\end{equation}
\end{proof}

\begin{cor}\label{cor:nolight}
 Let $\GL(N,q)$ act on $V^{\otimes q^N}$ by permutation. Then for any non-identity element $g \in \GL(N,q)$ the conformal weight of the $g$-twisted sector is bounded from below by
 \begin{equation}
     \rho_g > \frac{(q-1)Nc}{32}.
 \end{equation}
\end{cor}
\begin{proof}
Let $\{v_i\}$ be a basis of $\mathbb{F}_q^N$. If $g \in \mathrm{GL}(N,q)$ is not the identity there is a non-invariant basis element $v_1$. Now for any other basis element $v_i$ either $v_i$ is  non-invariant or $v_1 + v_i$ is non-invariant. In particular, if any element $v$ is not invariant, neither is $\alpha v$ for any $\alpha \in \mathbb{F}_q^*$. Thus $\mathbb{F}_q^N$ contains at least $(q-1)N$ non-invariant elements and the desired result follows from Corollary \ref{cor:rho_min}.
\end{proof}
This establishes theorem~\ref{Vorblimit}.

Finally, let us briefly discuss the relevance of the $\GL(N,q)$ orbifold for the growth of $b_n$. For general permutation orbifolds, the twisted modules tend to give exponential growth. More precisely, if there is  a conjugacy class in $G_N$ of total non-trivial cycle length $n$, then usually one gets  $\log b_n \sim 2\pi n$ \cite{Belin:2014fna}. This is in particular what happens for the $S_N$ orbifold \cite{Keller:2011xi}. The $\GL(N,q)$ orbifold circumvents this because all its non-trivial conjugacy classes $g$ give modules which satisfy corollary~\ref{cor:nolight}. However, despite not having any contributions from the twisted modules, the $b_n(\GL(N,q))$ still grow faster than the $b_n(S_N)$, as can be seen in figure~\ref{f:bn}. Nonetheless, the $\GL(N,q)$ represents a new growth behavior, different from the ones found so far in the literature \cite{Keller:2017rtk}.

\begin{figure}
	\begin{center}
\includegraphics[width=.7\textwidth]{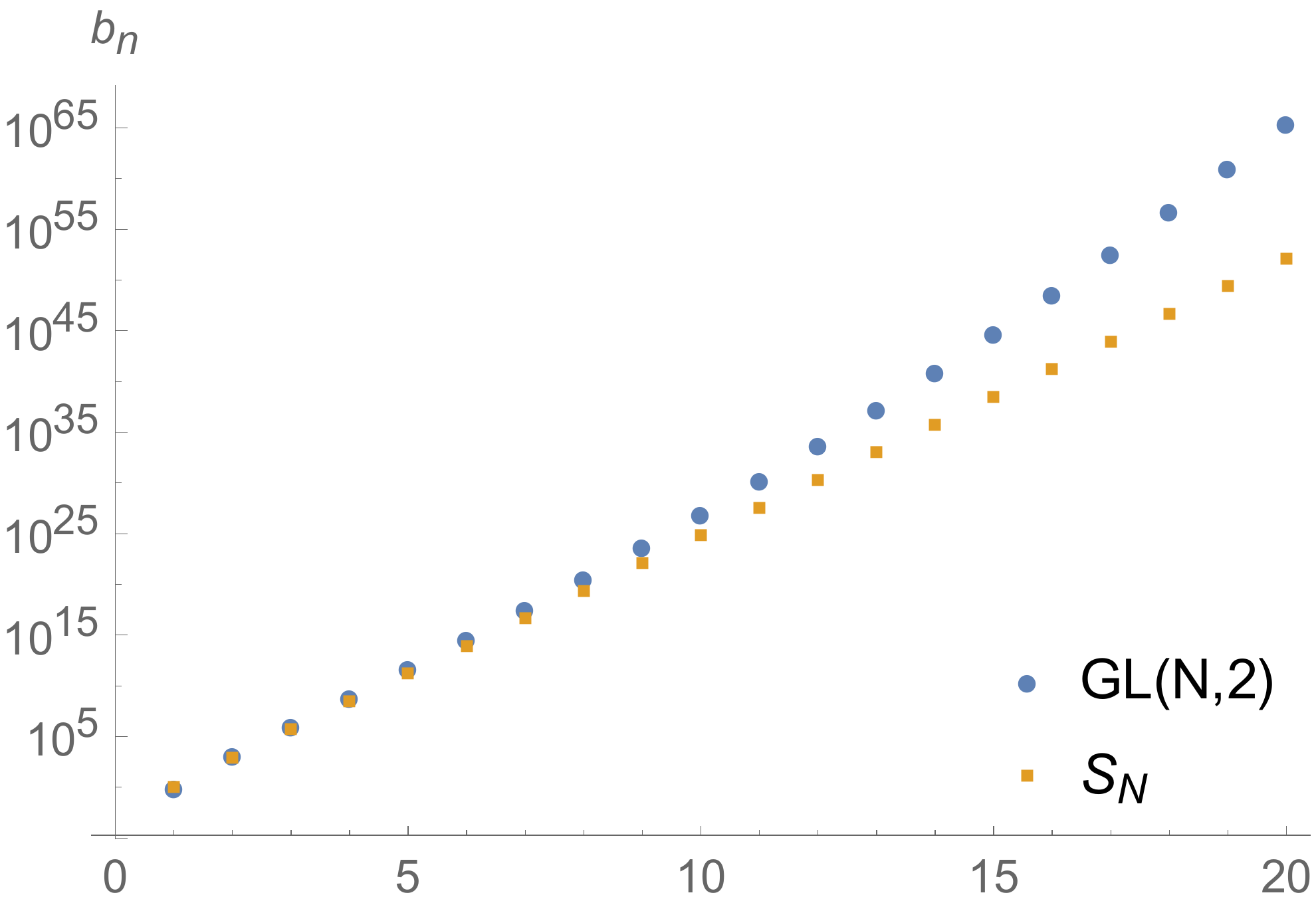}
	\end{center}
\caption{Growth of $b_n$ for $\Vi^{orb}$ for $S_N$ vs $\GL(N,2)$. For concreteness we chose $V=V_{E_8}^3$.}
\label{f:bn}
\end{figure}

\medskip

\bibliographystyle{alpha}
 \bibliography{./refmain}

\end{document}